\documentclass[11pt,fleqn]{article}

\usepackage{amsmath,amsfonts,amsthm,amssymb}
\usepackage{epsfig}

\newcommand{\R}{\mathbb R}
\newcommand{\C}{\mathbb C}

\newcommand{\beqr}{\begin{eqnarray*}}
\newcommand{\eeqr}{\end{eqnarray*}}

\renewcommand{\Re}{\mathop {\rm Re}\nolimits}

  \newtheorem{theorem}{Theorem}[section]
  \newtheorem{corollary}[theorem]{Corollary}
  \newtheorem{lemma}[theorem]{Lemma}

  \theoremstyle{remark}
  \newtheorem{rem}[theorem]{Remark}
\newtheorem{question}[theorem]{Question}

\begin{document}

\title{$K$-spectral sets and intersections of disks of the Riemann sphere}

\author{Catalin {\sc Badea}, Bernhard{ \sc Beckermann} and Michel {\sc Crouzeix} }

%\date{\today}
\date{}

\maketitle

\begin{abstract}

We prove that if two closed disks $X_1$ and $X_2$ of the Riemann sphere
are spectral sets for a bounded linear operator $A$ on a
Hilbert space, then $X_1\cap X_2$ is a complete $(2+2/\sqrt{3})$-spectral
set for $A$. When the intersection $X_1\cap X_2$ is an annulus, this
result gives a positive answer to a question of A.L.~Shields (1974).
\end{abstract}

\section{Introduction and the statement of the main results.}
Let $X$ be a closed set in the complex plane and let $R(X)$ denote
the algebra of bounded rational functions on $X$, viewed as a
subalgebra of $C(\partial X)$ with the supremum norm $$\|f\|_X =
\sup \{ |f(x)| : x \in X\} = \sup \{ |f(x)| : x \in
\partial X\}. $$ Here $\partial X$ denotes the boundary of the set
$X$.

\subsection{Spectral and complete spectral sets.} Let $A \in \mathcal{L}(H)$
 be a bounded linear operator acting on a complex Hilbert space $H$.
 For a fixed constant $K > 0$, the set $X$ is said to be a
 $K$-\emph{spectral} set for $A$ if the spectrum $\sigma(A)$ of $A$
 is included in $X$ and the inequality $\|f(A)\| \le K\|f\|_X$ holds
 for every $f \in R(X)$. Notice that, for a rational
 function $f=p/q \in R(X)$,
 the poles of $f$ are outside of $X$, and the operator $f(A)$ is
 naturally defined as $f(A) = p(A)q(A)^{-1}$ or, equivalently,
 by the Riesz holomorphic functional
 calculus. The set $X$ is a \emph{spectral} set
 for $A$ if it is a $K$-spectral set with $K=1$. Thus $X$ is spectral
 for $A$ if and only if $\|\rho\| \le 1$,
 where $\rho : R(X) \mapsto \mathcal{L}(H)$ is the homomorphism
 given by $\rho(f) = f(A)$.

We let $M_n(R(X))$ denote the algebra of $n$ by $n$ matrices with entries from $R(X)$. If we let the $n$ by $n$ matrices have the operator norm that they inherit as linear transformations on the $n$-dimensional Hilbert space $\C^{n}$, then we can endow $M_n(R(X))$ with the norm
$$\|\left(f_{ij}\right)\|_X = \sup \{ \|\left(f_{ij}(x)\right)\| : x \in X\} = \sup \{ \|\left(f_{ij}(x)\right)\| : x \in \partial X\}. $$
In a similar fashion we endow $M_n(\mathcal{L}(H))$ with the norm it inherits by regarding an element $(A_{ij})$ in $M_n(\mathcal{L}(H))$ as an operator acting on the direct sum of $n$ copies of $H$. For a fixed constant $K > 0$, the set $X$ is said to be a \emph{complete} $K$-\emph{spectral} set for $A$ if $\sigma(A) \subset X$ and the inequality $\|(f_{ij}(A))\| \le K\|(f_{ij})\|_X$ holds for every matrix $(f_{ij}) \in M_n(R(X))$ and every $n$. In terms of the complete bounded norm (\cite{paul}) of the homomorphism $\rho$, this means that $\|\rho\|_{cb} \le K$. A \emph{complete spectral} set is a complete $K$-spectral set with $K=1$.

Spectral sets were introduced and studied by J. von Neumann \cite{Neu} in 1951. In the same paper von Neumann proved that a closed disk $\{z \in \C : |z-\alpha| \le r\}$ is a spectral set for $A$ if and only if $\|A -\alpha I \| \le r$. Also \cite{Neu}, the closed set $\{z \in \C : |z-\alpha| \ge r\}$ is spectral for $A \in \mathcal{L}(H)$ if and only if $\|(A -\alpha I)^{-1} \| \le r^{-1}$. We refer to two books \cite{conway, paul} for a survey of known properties of spectral and complete spectral sets.

\subsection{The annulus as a $K$-spectral set} Let $r$ and $R$ be
two positive constants with $r < R$. Let $A \in \mathcal{L}(H)$ be
an invertible operator such that $\|A\| \le R$ and $\|A^{-1}\| \le
1/r$. Then $X_1 = \{z \in \C : |z| \le R\}$ and $X_2 = \{z \in \C
: |z| \ge r\}$ are spectral sets for $A$. The annulus $$X(r,R) =
\{z \in \C : r \le |z| \le R\} = X_1 \cap X_2$$ is not necessarily
spectral for a given invertible operator $A$. Examples can be
found in \cite{williams, misra, paulsen}. Given an invertible
operator $A$ with $\|A\| \le R$ and $\|A^{-1}\| \le 1/r$, Shields
proved in \cite{shields} that $X(r,R)$ is a $K$-spectral set for $A$ with $K =
2+ \left(\left( R+r\right)/\left(R-r\right)\right)^{1/2}$. The
following questions were asked by Shields (see \cite[Question
7]{shields}):

\begin{question}\label{q1} Find the best constant $K(r,R)$, i.e.,
the smallest constant $C$ such that $X(r,R)$ is a $C$-spectral set
for all invertible $A\in \mathcal L(H)$ with $\| A \| \le R$ and
$\| A^{-1} \|\leq r^{-1}$.
\end{question}
\begin{question}\label{q2}
 Fixing (for instance) $R$, is this best constant bounded (as a function of $r$) ?
\end{question}

In analogy with Question~\ref{q1}, we will denote by $K_{cb}(r,R)$
the smallest constant $C$ such that $X(r,R)$ is a complete
$C$-spectral set. The same proof of Shields (see also \cite{dopa,
paul}) shows that in fact $K_{cb}(r,R) \le 2+ \left(\left(
R+r\right)/\left(R-r\right)\right)^{1/2}$.
\subsection{Statement of the main results.} The aim of the present note
is to study the intersection of two closed disks of the Riemann sphere which
are spectral sets for a Hilbert space bounded linear operator. In the
case of the annulus we
give an estimate for $K(r,R)$ (a partial answer to Question \ref{q1})
which allows to give a positive answer to Question \ref{q2}.

We describe now the main results of this paper. By possibly
multiplying the operator by a scalar, we see that
$K(r,R)=K(\sqrt{r/R},\sqrt{R/r})$. This allows to assume, 
%and hence we can and we will
without any loss of generality, that $r = R^{-1}$. We have
the following result.

\begin{theorem}\label{thm1}
   Let $R > 1$, $X = X(R^{-1},R)= \{z \in
\C : R^{-1} \le |z| \le R\}$, and denote by $K(R) = K(R^{-1},R)$
(and $K_{cb}(R) = K_{cb}(R^{-1},R)$, respectively), the smallest
constant $C$ such that $X$ is a $C$-spectral set (and a complete
$C$-spectral set, respectively) for any invertible $A \in
\mathcal{L}(H)$ verifying $\|A\| \le R$ and $\|A^{-1}\| \le R$.
Then \beqr \frac{2}{1+R^{-2}} & < & K(R) \le K_{cb}(R) \\
 &\le &
 2 + \min\left( \sqrt{\frac{R^2+2R+1}{R^2+R+1}} ,
 \sqrt{\frac{R^2+1}{R^2-1}}\right) \leq 2 + \frac{2}{\sqrt{3}} < 3.2.
\eeqr
\end{theorem}

In particular $K(R)$ and $K_{cb}(R)$ are bounded functions of $R$.
We obtain the following consequence about normal dilations.

\begin{corollary}\label{cor1}
 Let $R > 1$. Let $A \in \mathcal{L}(H)$ be an invertible operator
verifying $\|A\| \le R$ and $\|A^{-1}\| \le R$. Let $X = \{z \in \C : R^{-1} \le |z| \le R\}$. Then there exist an invertible operator $L \in \mathcal{L}(H)$ with $\|L\|\cdot\|L^{-1}\| \le 2+2/\sqrt{3}$, a larger Hilbert space $\mathcal{H} \supset H$ and an invertible normal operator $N \in \mathcal{L}(\mathcal{H})$ with $\sigma(N) \subset \partial X$ such that
$$ L^{-1}f(A)L = P_Hf(N)\mid_H \quad (f \in R(X)) .$$
Here $P_H$ is the orthogonal projection of $\mathcal{H}$ onto $H$.
\end{corollary}

Besides the annulus, (complete) $K$-spectral sets which are
intersections of spectral disks of the complex plane have been
considered in \cite{stampfliI, stampfliII, lewis, crde, becr} ;
we refer to \cite{becr} for a discussion of the best possible
constant $K$. In the second part of our paper we consider the more
general case of intersection of two closed disks $X_1$ and $X_2$ of the
Riemann sphere. We prove the following result.

\begin{theorem}\label{thm3}
Let $X_1$ and $X_2$ be two closed disks of the
Riemann sphere. If $X_1$ and $X_2$ are spectral sets for a bounded
operator $A$ in a Hilbert space, then $X_1\cap X_2$ is a complete
$(2+2/\sqrt{3})$-spectral set for $A$.
\end{theorem}
This theorem extends previously known results concerning the
intersection of two disks in $\C$ to not necessarily convex or
simply connected $X_1\cap X_2$. Note that the case of finitely
connected compact sets has been studied in \cite{dopa, paul},
however, without a uniform control on the constant $K$.

Note also that, if we consider two distinct bounded, convex and
closed subsets $X_1$ and $X_2$ of the complex
plane, and if we assume that  $X_1$ and $X_2$ are spectral sets
for $A$, then 
%it follows from \cite{crzx} that 
$X_1\cap X_2$ is a
complete $11.08$-spectral set for $A$. Indeed, the fact that $X_j$
is a spectral set for $A$ implies that the numerical range
$W(A)=\{\langle Ax, x\rangle : \|x\| = 1\}$ is included in $X_j$, $j=1,2$, and according to \cite{crzx} the closure of the numerical range $W(A)$ is a complete $11.08$-spectral set for $A$. However, the result from \cite{crzx} does not imply a solution of Shields' Question \ref{q2}. We refer also to \cite{PS, bcd, crzx} for some normal dilation results for the numerical range, in the spirit of Corollary \ref{cor1}. 

% normal dilation, on the numerical range of an arbitrary Hilbert space operator.
% 
% statements of the 
% 
% We refer to \cite{PS, BCD, crzx} and the references cited therein for studies of constants $K$ for which the numerical range is a (complete) $K$-spectral set and for the corresponding normal dilation consequences. The constant $11.08$ from \cite{crzx} is the first universal bound~; it is conjectured in \cite{crzx} that the constant $11.08$ can be decreased to $2$.
% 
% $2$ is the best this can be decreased, as well to \cite{crzx} where the first universal bound 

% %%%
% Note also that, if we consider two distinct convex and
% closed subsets $X_1$ and $X_2$ of the complex
% plane, and if we assume that  $X_1$ and $X_2$ are spectral sets
% for $A$, then $X_1\cap X_2$ is a
% complete $M$-spectral set for $A$ for a certain constant $M$ (for instance $M=11.08$).
% 
% %complete $K$-spectral set for $A$ for a certain constant $K$. 
% 
% Indeed, the fact that $X_j$
% is a spectral set for $A$ implies that the numerical range
% $W(A)=\{\langle Ax, x\rangle : \|x\| = 1\}$ is included in $X_j$ ($j=1,2$). We refer to \cite{PS, BCD, crzx} and the references cited therein for studies of constants $M$ for which the numerical range is a (complete) $M$-spectral set and for the corresponding normal dilation consequences. The constant $11.08$ obtained in \cite{crzx} is the first universal bound for $M$. ~; it is conjectured in \cite{crzx} that the constant $11.08$ can be decreased to $2$.
% 

%%%%%%%%%%%%%%
The remainder of the paper is organized as
follows: we first show in \S 2 that Theorem \ref{thm1} together
with some results from \cite{crde,becr} implies Theorem
\ref{thm3}. Our proof of Theorem \ref{thm1} is based on a
representation formula for $f(A)$ established in \S 3. Finally,
the proofs of Theorem~\ref{thm1} and Corollary~\ref{cor1} are
provided in \S 4.

\section{Proof of Theorem \ref{thm3} using Theorem \ref{thm1}}
Let $X_1$ and $X_2$ be two closed disks of the Riemann sphere,
which are spectral sets for a bounded linear operator $A$ in a
Hilbert space. Here six different situations have to be
considered, see Figure~\ref{figure1}.

\begin{figure}[t]
 \centerline{\epsfig{file=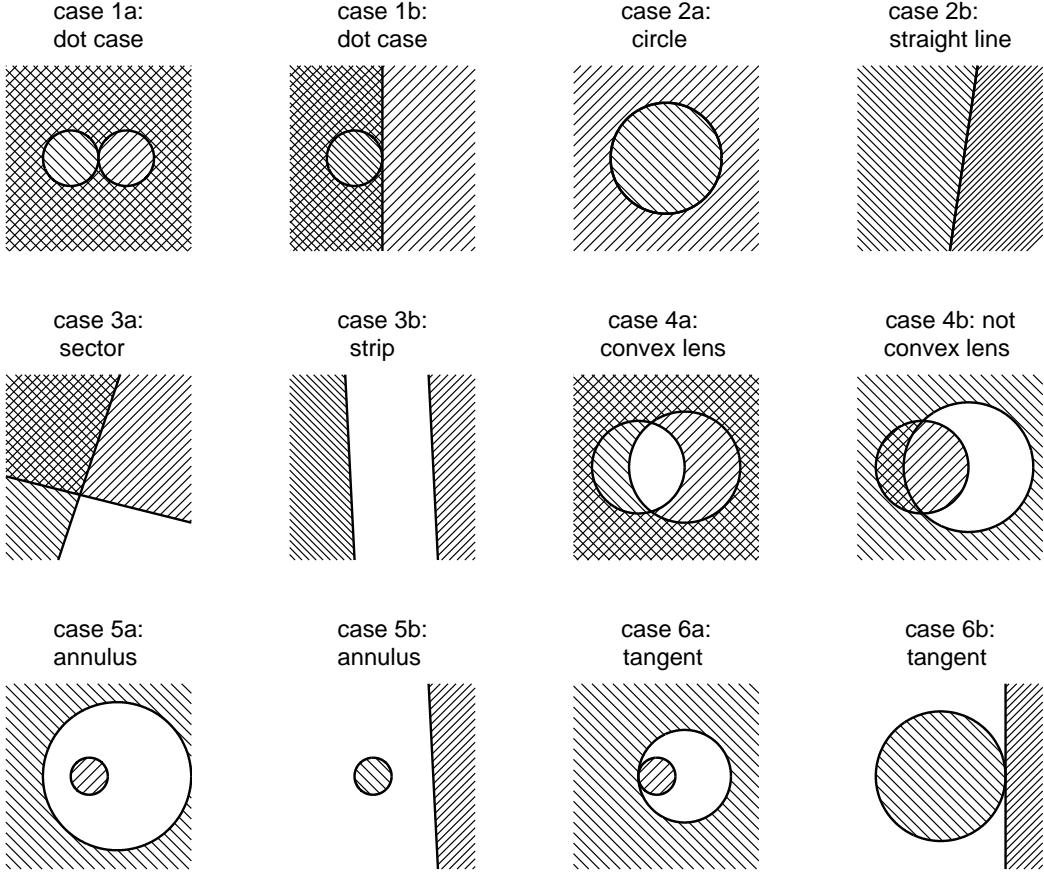,scale=.9}}
 \caption{{\it The six different cases occurring by considering
 intersections of closed disks on the Riemann
 sphere.}}\label{figure1}
\end{figure}

\begin{description}
\item[Case 1:] $X_1\cap X_2=\{\lambda\}$ is a singleton.
Then we have $A=\lambda I$ and
$X_1\cap X_2$ clearly is a complete spectral set for $A$.

\item[Case 2:] $X_1\cap X_2$ is a circle or a straight line.
Then $A$ is a normal operator with spectrum $\sigma(A)$ contained
in $X_1\cap X_2$. This yields that  $X_1\cap X_2$ is a complete
spectral set for $A$.

\item[Case 3:] $X_1\cap X_2$ is a convex sector or a strip of the complex plane.
In this case, both $X_1$ and $X_2$ are half-planes, and a closed
half-plane $\Pi$ is a
 spectral set for $A$ if and only if the numerical range $W(A)$ is a subset of $\Pi$.
Thus $W(A) \subset X_1 \cap X_2$.
 It follows from \cite{crde} that $X_1\cap X_2$ is a complete
$K$-spectral set, with $K\leq 2+2/\sqrt3$.

 \item[Case 4:] $\partial X_1\cap \partial X_2=\{\lambda_1,\lambda_2\}$ is a set
 consisting of two distinct points of $\C$. Here $X_1
 \cap X_2$ is lens-shaped. If it is in addition convex, then from
 \cite{becr} we know that $X_1 \cap X_2$ is a complete $K$-spectral set,
 with $K\leq 2+2/\sqrt3$. The proof for not convex lenses is the
 same, we repeat here the main idea for the sake of completeness.
 Let us first assume
 that $\lambda_1\notin \sigma(A)$ and
 set $B=\varphi(A)$ with $\varphi(z)=(\lambda_1\!-\!z)^{-1}$
 and $Y_j=\varphi(X_j)$, $j=1,2$.
 Then both $Y_j$ are closed half-planes.
The von Neumann inequality for disks shows that $Y_j$ are spectral
sets for $B$, see also \cite[\S~154, Lemma~2]{riesz}. It follows
from the previous case that $Y_1\cap Y_2$ is a complete
$K$-spectral set for $B$ and thus $X_1 \cap X_2$ is a complete
$K$-spectral set for $A$, with the same constant $K$. Finally, if
$\lambda_1 \in \sigma(A)$, we can replace the disk $X_1$ of the
Riemann sphere, of radius $R_1$, by a concentric disk $X_1'\supset
X_1$, of radius $R_1\pm\varepsilon$. Then, for $\varepsilon>0$
small enough, $\partial X'_1\cap
\partial X_2=\{\lambda'_1,\lambda'_2\}$ is still a set with two
distinct points of $\C$, the set $X'_1$ is a spectral set for $A$
and $\lambda'_1\notin \sigma(A)$. We conclude that $X_1\cap X_2$
is a complete $K$-spectral set for $A$ by letting $\varepsilon\to
0$.

 \item[Case 5:] $\partial X_1\cap \partial X_2=\emptyset$, but $X_1\cap
 X_2$ is not a strip. For the special case
 $X_1\cap X_2=\{z\in \C\,; R^{-1}\leq |z|\leq R\}$, $R>1$,
Theorem \ref{thm1} implies that
 $X_1\cap X_2$ is a complete $(2+2/\sqrt{3})$-spectral set for $A$.
 In the general case, we may find $R>1$ and a linear fractional
 transformation $\varphi$ such that $\varphi(X_1)=\{z\in \C\,; |z|\leq R\}$
 and $\varphi(X_2)=\{z\in \C\,; |z|\geq R^{-1}\}$.
 Then, setting $B=\varphi(A)$ and $Y_j=\varphi(X_j)$, $j=1,2$, we have
 that $Y_j$ is a spectral set for $B$, see also \cite[\S~154, Lemma~2]{riesz}.
 Thus $\{z\in \C\,; R^{-1}\leq |z|\leq R\}=\varphi(X_1\cap X_2)$
 is a complete $(2+2/\sqrt 3)$-spectral set for $B$, which is
 equivalent to $X_1\cap X_2$ is a complete $(2+2/\sqrt 3)$-spectral
 set for $A$.
 \item[Case 6:] $\partial X_1\cap \partial X_2=\{\lambda\}$ is reduced to a
 single point, but $X_1\cap X_2$ is neither a singleton, nor a sector nor a strip.
%not a singleton, nor a sector nor a strip. 
In this case
at least one of the sets $X_j$, $j=1,2$, is the interior or the exterior of a disk and the boundaries of the sets $X_j$ are tangent in one point. We can replace the disk, say $X_1$,
 of radius $R_1$, by a concentric disk  $X_1'\supset X_1$,
 of radius $R_1\pm\varepsilon$. Then,
 for $\varepsilon>0$ small enough,
 $\partial X'_1\cap \partial X_2=\emptyset$, and we obtain from
 the previous case that $X_1\cap X_2$ is a complete $K$-spectral set
 for $A$ by letting $\varepsilon\to 0$.
 \end{description}

\section{A decomposition lemma for annuli}
In order to give a proof of the upper bound of Theorem \ref{thm1} we need the following representation formula for $f(A)$.

\begin{lemma}\label{le}
Let $A \in \mathcal{L}(H)$ be an operator satisfying $\|A\|<R$ and $\|A^{-1}\|<R$. We set $r=1/R$ and denote by $X$ the annulus $X=X(R^{-1},R)=\{z\in \C\,; r\leq |z|\leq R\}$. For any bounded rational function $f$ on $X$, we have the representation formula
\begin{equation*}f(A)=\int_0^{2\pi}\!\!  f(Re^{i\theta})\,\mu(\theta,A)\,d\theta+
\int_0^{2\pi} \!\!  f(re^{i\theta})\,\mu(-\theta,A^{-1})\,d\theta+
\int_0^{2\pi}\!\!  f(e^{i\theta})\,M(\theta,A^*)^{-1}\,d\theta,
\end{equation*}
where \hskip-.4cm
\begin{align*}
\mu(\theta,A)&=\frac{1}{4\pi}\big((1\!+\!e^{-i\theta}r A)(1\!-\!e^{-i\theta}r A)^{-1}+(1\!+\!e^{i\theta}r A^*)(1\!-\!e^{i\theta} rA^*)^{-1}\big),  \quad\hbox{and}\\
M(\theta, A^*)&=\frac{2\pi}{R^2-r^2}(R^2+r^2-(e^{i\theta}A^*)^{-1}-e^{i\theta}A^*).
\end{align*}
\end{lemma}
\begin{proof} We get from the Cauchy formula
\begin{equation*}
f(A)=\tfrac{1}{2\pi i}\int_{\partial X}f(\sigma)\,((\sigma\!-\!A)^{-1}\,d\sigma-
(\bar\sigma\!-\!A^*)^{-1}\,d\bar\sigma)
+\tfrac{1}{2\pi i}\int_{\partial X}f(\sigma)\,
(\bar\sigma\!-\!A^*)^{-1}\,d\bar\sigma=F_{1}+F_{2}.
\end{equation*}
Let us set $\Gamma_{\rho}=\{\rho\,e^{i\theta};\theta\in [0,2\pi]\}$.
The part $\Gamma_R$ of $\partial X$ is counterclockwised oriented and, with $\sigma=Re^{i\theta}$, we have
\begin{align*}
\tfrac{1}{2\pi i}((\sigma\!-\!A)^{-1}\,d\sigma-
(\bar\sigma\!-\!A^*)^{-1}\,d\bar\sigma)
&=\tfrac{1}{2\pi }((Re^{i\theta}\!-\!A)^{-1}\,Re^{i\theta}+
(Re^{-i\theta}\!-\!A^*)^{-1}\,Re^{-i\theta})\,d\theta\\
&=\tfrac{1}{2\pi }((1\!-\!e^{-i\theta}r A)^{-1}+(1\!-\!e^{i\theta}r A^*)^{-1})\,d\theta\\
&=\tfrac{1}{2\pi}\,d\theta+\mu(\theta,A)\,d\theta.
\end{align*}
The other component $\Gamma_r$ is clockwised oriented and, with $\sigma=re^{i\theta}$, we have
\begin{align*}
\tfrac{1}{2\pi i}((\sigma\!-\!A)^{-1}\,d\sigma-
(\bar\sigma\!-\!A^*)^{-1}\,d\bar\sigma)
&=\tfrac{1}{2\pi }((re^{i\theta}\!-\!A)^{-1}\,re^{i\theta}+
(re^{-i\theta}\!-\!A^*)^{-1}\,re^{-i\theta})\,d\theta\\
&=\tfrac{1}{2\pi}\,d\theta-\mu(-\theta,A^{-1})\,d\theta.
\end{align*}
Noticing that $\int_0^{2\pi}f(Re^{i\theta})\,d\theta=\int_0^{2\pi}f(re^{i\theta})\,d\theta$, we obtain that
\begin{equation*}
F_{1}=\int_0^{2\pi}f(Re^{i\theta})\,\mu(\theta,A)\,d\theta+
\int_0^{2\pi}f(re^{i\theta})\,\mu(-\theta,A^{-1})\,d\theta.
\end{equation*}
We consider now the second term $F_2$. On the component $\Gamma_{R}$ we have $\bar \sigma=R^2/\sigma$,
and thus
\begin{align*}
\tfrac{1}{2\pi i}\int_{\Gamma_R}f(\sigma)\,(\bar\sigma\!-\!A^*)^{-1}\,d\bar\sigma
&=-\tfrac{1}{2\pi i}\int_{\Gamma_R}f(\sigma)\,
(R^2\!-\!\sigma A^*)^{-1}\,\frac{R^2}{\sigma}\,d\sigma\\
&=-\tfrac{1}{2\pi i}\int_{\Gamma_1}f(\sigma)\,
(R^2\!-\!\sigma A^*)^{-1}\,\frac{R^2}{\sigma}\,d\sigma.
\end{align*}
Indeed, the last integrand is holomorphic in $\sigma$. Hence we
can replace the integration path $\Gamma_R$ by $\Gamma_1$
(counterclockwised oriented). We similarly have for the second
component
\begin{align*}
\tfrac{1}{2\pi i}\int_{\Gamma_r}f(\sigma)\,(\bar\sigma\!-\!A^*)^{-1}\,d\bar\sigma
=\tfrac{1}{2\pi i}\int_{\Gamma_1}f(\sigma)\,
(r^2\!-\!\sigma A^*)^{-1}\,\frac{r^2}{\sigma}\,d\sigma
\end{align*}
by taking into account the opposite orientation of $\Gamma_r$.
Therefore
\begin{align*}
F_2&=\tfrac{1}{2\pi i}\int_{\Gamma_1}f(\sigma)\,
\big((r^2\!-\!\sigma A^*)^{-1}\,\frac{r^2}{\sigma}-(R^2\!-\!\sigma A^*)^{-1}\,\frac{R^2}{\sigma}\big)
\,d\sigma\\
&=\int_0^{2\pi}f(e^{i\theta})\,M(\theta,A^*)^{-1}\,d\theta,
\end{align*}
which completes the proof of the lemma.
\end{proof}

\section{ The complete bound in an annulus}

We keep the notation from the previous section. The
following lemma shows that $\Re M(\theta, A^*)$ is a positive
operator.

\begin{lemma} \label{Lemma 41}
Assume that $\|A\| < R$ and $\|A^{-1}\| < R$. Let $r = R^{-1}$.
Then we have the lower bound
\begin{equation*}
\Re M(\theta, A^*)\geq N(\theta):=\frac{2\pi}{R^2-r^2}
   \left((R^2+r^2-R-r) + \frac{R+r+2}{4}
   \left(2-e^{i\theta}U^*-e^{-i\theta}U\right)\right),
\end{equation*}
where $U$ denotes the unitary operator such that $A=U G$, with $G$
self-adjoint positive definite. Also, $N(\theta)$ is a positive
invertible operator.
\end{lemma}
\begin{proof}We have
\begin{align*}
  \frac{R^2-r^2}{2\pi}\Re M(\theta, A^*)&=
   R^2+r^2-\Re((e^{-i\theta}A)^{-1}+e^{i\theta}A^*)\\&=
   R^2+r^2 - \Re\big(e^{i\theta}(G^{-1}+G )U^*\big)\\&=
   R^2+r^2 - \tfrac{R+r+2}{2} \Re\big(e^{i\theta} U^*\big)
   - \Re\big(e^{i\theta}(G^{-1}+G-\tfrac{R+r+2}{2} )U^*\big)
\end{align*}
We note that the assumptions $\|A\|\leq R$ and $\|A^{-1}\|\leq R$
are equivalent to $\|G\|\leq R$ and $\|G^{-1}\|\leq R$. Since $G$
is self-adjoint, this means that $r\leq G \leq R$, and hence
\begin{equation*}
\|  G^{-1}+G-\tfrac{R+r+2}{2} \|\leq \sup_{r\leq x\leq R} |
x^{-1}+x-\tfrac{R+r+2}{2}| = \tfrac{R+r-2}{2} .
\end{equation*}
It follows that
\begin{align*}
  \frac{R^2-r^2}{2\pi}\Re M(\theta, A^*)&\geq
   R^2+r^2 - \tfrac{R+r+2}{2} \Re\big(e^{i\theta} U^*\big)
   - \tfrac{R+r-2}{2}
   \\&= R^2 + r^2 - R - r + \tfrac{R+r+2}{2} \Re\big(1-e^{i\theta}
   U^*\big),
\end{align*}
which completes the proof of the lemma.
\end{proof}

\begin{proof}[Proof of the upper bound of Theorem \ref{thm1}]
We can suppose that $\|A\| < R$ and $\|A^{-1}\| < R$. Using the
notation of Lemma~\ref{le}, it follows from the condition $\|A\|<
R$ that $\mu(\theta,A )\geq 0$ for all $\theta\in \R$. Therefore
we have
\begin{equation*}
\Big\| \int_0^{2\pi}\!\!
f(Re^{i\theta})\,\mu(\theta,A)\,d\theta\Big\|\leq \Big\|
\int_0^{2\pi}\!\!  \mu(\theta,A)\,d\theta\Big\|\ \|f\|_{X}
=\|f\|_{X} .
\end{equation*}
Here we have used that $\int_0^{2\pi}\!\!
\mu(\theta,A)\,d\theta=1$, which follows from the residue formula.
Similarly we have $\mu(-\theta,A^{-1})\geq0$ and we get the
estimate
\begin{equation*}
\Big\| \int_0^{2\pi} \!\!
f(re^{i\theta})\,\mu(-\theta,A^{-1})\,d\theta\Big\| \leq \|f\|_{X}
.
\end{equation*}
Using Lemma \ref{le} and the positivity of $\Re M(\theta,A^*)$ for all $\theta\in \R$ (Lemma \ref{Lemma 41}) we obtain the estimate
\begin{equation*}
\|f(A)\|\leq K\,\|f\|_{X},\quad \hbox{with}\quad
K=2+\Big\|\int_0^{2\pi}(\Re M(\theta, A^*))^{-1}\,d\theta\Big\|.
\end{equation*}
Let $\rho : R(X) \mapsto \mathcal{L}(H)$ be the homomorphism given by $\rho(f) = f(A)$. Therefore the norm of $\rho$ is bounded by $K$. Furthermore, since we only have used arguments based on positivity of operators, it is easily seen that the complete bounded norm $\|\rho\|_{cb}$ is also bounded by $K$.

Taking into account the bound of Shields \cite{shields}, for establishing the
upper bound of Theorem~\ref{thm1} it suffices now to show that
\begin{equation}\label{rouge}
  \Big\|\int_0^{2\pi}(\Re M(\theta,A^*))^{-1}\,d\theta\Big\|
  \leq \sqrt{\frac{R^2+2R+1}{R^2+R+1}} \leq \frac{2}{\sqrt{3}}.
\end{equation}
Consider the function
\begin{equation*}
   J(z):=\frac{R^2-r^2}{2\pi}\int_{0}^{2\pi}
   \left((R^2+r^2-R-r) + \frac{R+r+2}{4}
   \left(2-e^{i\theta}z^{-1} - e^{-i\theta}z\right)\right)^{-1}
   \, d\theta.
\end{equation*}
Since $U$ is a unitary operator, it follows from Lemma \ref{Lemma
41} that
\begin{equation*}
\Big\|\int_0^{2\pi}(\Re M(\theta, A^*))^{-1}\,d\theta\Big\| \leq
\Big\|\int_0^{2\pi}(N(\theta))^{-1}\,d\theta\Big\| = \| J(U) \|
 = \sup \left\{ | J(e^{i\phi}) | : e^{i\phi} \in \sigma(U) \right\} .
\end{equation*}
On the other hand, we have
\begin{align*}
J(e^{i\varphi})
  &=\frac{R^2-r^2}{2\pi}
   \int_{0}^{2\pi}
   \frac{1}{(R^2+r^2-R-r) + \tfrac{R+r+2}{4} (2- 2\cos(\theta\!-\!\varphi))}
   \,d\theta\\
   &=\frac{R^2-r^2}{2\pi}
   \int_{-\infty}^{\infty}
   \frac{2}{(R^2+r^2-R-r)(1+s^2) + (R+r+2) s^2}
   \,ds\\
   &=\frac{R^2-r^2}{2\pi}
   \int_{-\infty}^{\infty}
   \frac{2}{(R^2+r^2-R-r)+  (R^2+r^2+2) s^2}
   \,ds\\ &=
   \sqrt{\frac{R^2+2R+1}{R^2+R+1}} \, = \, \sqrt{\frac{1}{1-\frac{1}{(\sqrt{R}+1/\sqrt{R})^2}}}
   \, \leq \, \frac{2}{\sqrt{3}},
\end{align*}
which implies (\ref{rouge}). This gives a proof of the upper bound of
Theorem \ref{thm1} for $K_{cb}(R)$.
\end{proof}

%%

%\medskip

\begin{proof}[Proof of the lower bound of Theorem \ref{thm1}]
 For $t \in \C$, let $A(t) = \left(\begin{array}{cc}
            1 & t  \\
            0 & 1
        \end{array}\right) \
$ with inverse $A(t)^{-1} = \left(\begin{array}{cc}
            1 & -t  \\
            0 & 1
        \end{array}\right) \ $
acting on the Hilbert space $\C^2$. For $t_0 = R - R^{-1}$ we have
$\|A(t_0)\| = \|A(t_0)^{-1}\| = R$ (compare with \cite[p. 152]{paul}). We
will make use of the following result from geometric function theory about
the
infinitesimal Carath\'eodory metric: it is shown by Simha in
\cite[Example~(5.3)]{simha} that
\begin{equation*}
    \sup \left\{ \frac{| f'(1) |}{\|f\|_{X}} :
    \mbox{$f$ analytic in $X$ and $f(1)=0$} \right\}
    = \frac{2}{R} \prod_{n=1}^\infty \Bigl(\frac{1-R^{-8n}}{1-R^{4-8n}}
    \Bigr)^2,
\end{equation*} with the supremum being attained for some function
$f_0$ analytic in $X$, with $\|f_0\|_{X}=1$ and $f_0(1)=0$.
Therefore
\begin{eqnarray*}%\label{eq2}
    K(R) &\geq& \frac{1}{\| f_0 \|_X} \| f_0(A(t_0)) \| =
    \left\|\left(\begin{array}{cc}
            f_0(1) & t_0 f_0'(1)  \nonumber \\
            0 & f_0(1)
        \end{array}\right)\right\| = t_0 \, | f_0'(1) | = \gamma(R)
\end{eqnarray*}
with \begin{eqnarray*}
     \gamma(R) &:=& 2 (1-R^{-2}) \prod_{n=1}^\infty
\Bigl(\frac{1-R^{-8n}}{1-R^{4-8n}}
    \Bigr)^2
    = \frac{2}{1+R^{-2}}
    \prod_{n=1}^\infty \frac{(R^{4n}-R^{-4n})^2}{(R^{4n}-R^{4-4n})(R^{4n}-R^{-4-4n})}
    \\&=& \frac{2}{1+R^{-2}}
    \prod_{n=1}^\infty \Bigl(1-\frac{(R^2-R^{-2})^2}{(R^{4n} -
    R^{-4n})^2}\Bigr)^{-1}
    .
%    \nonumber \\
%& > & \frac{2}{1+R^{-2}},
\end{eqnarray*}
This yields the estimate \begin{equation}\label{eq2}
    K(R) > \frac{2}{1+R^{-2}},
\end{equation} as claimed in Theorem~\ref{thm1}.
It remains to justify why we are allowed to take for a lower bound
of $K(R)$ the function $f_0$ which is not a rational function.
Indeed, by using instead of $f_0$ partial sums of the Laurent
expansion of an extremal function for the infinitesimal
Carath\'eodory metric on the annulus $1/R' < |z| < R'$ for some
$R'>R$ we obtain the same conclusion after taking the limit $R'
\to R$.
\end{proof}
%\medskip

\begin{figure}[t]
 \centerline{\epsfig{file=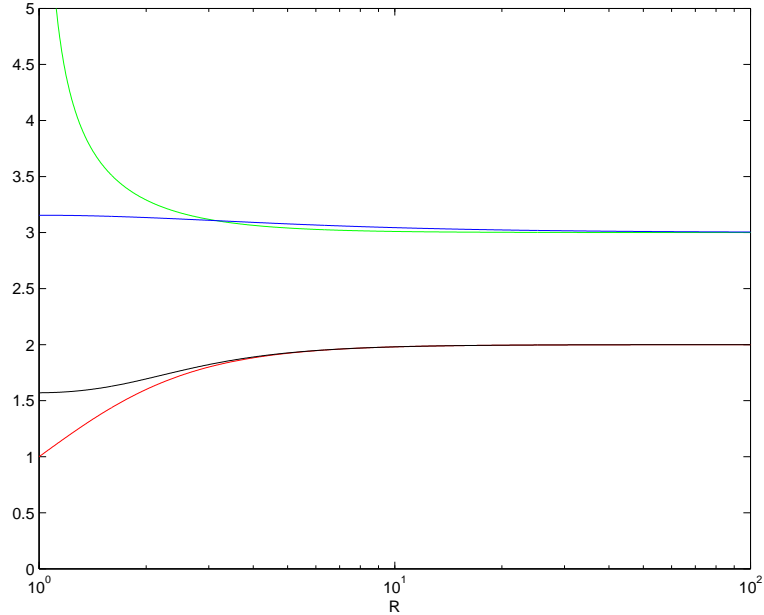,scale=.6}}
 \caption{{\it The two upper bounds and the lower bound for $K(R)$ from
 Theorem~\ref{thm1}, and the lower bound $\gamma(R)$
 from the proof of Theorem~\ref{thm1}.}}\label{figure2}
\end{figure}

\begin{rem}
  The final estimate (\ref{eq2}) of the preceding proof is not very sharp
  for $R$ close to one (see Figure~\ref{figure2}), and $\gamma(R)$ is
  a sharper but less readable lower bound for $K(R)$.
  For instance, for $R\to 1$ the lower bound
  $2/(1+R^{-2})$ of Theorem~\ref{thm1} tends to $1$ but
  %from the above proof we see that
  \begin{equation*}
      \lim_{R\to 1} \gamma(R)= \lim_{R\to 1}
      \prod_{n=1}^\infty \Bigl(1-\frac{(R^2-R^{-2})^2}{(R^{4n} -
    R^{-4n})^2}\Bigr)^{-1} = \prod_{n=1}^\infty
    \Bigl(1-\frac{1}{4n^2}\Bigr)^{-1} = \frac{\pi}{2} .
  \end{equation*}
In contrast, for our fixed matrix $A(t_0)$, it follows from \cite[Theorem 1]{FR} and \cite{simha} that the function $f_0$ is extremal within the class of functions analytic in $X$. 
% In contrast, the choice of the function $f_0$ is optimal since $f_0$ is an extremal function for the fixed matrix $A(t_0)$. Indeed, it follows from \cite[Theorem 1]{FR} and \cite{simha} that
% \begin{equation*}
%  \sup\left\{ \frac{1}{\|f\|_{X}}\left\| f(A(t))\right\| : \mbox{ f analytic in } X\right\} = \max\left( 1, |t|\frac{2}{R} \prod_{n=1}^\infty \Bigl(\frac{1-R^{-8n}}{1-R^{4-8n}}
%     \Bigr)^2\right)\
% \end{equation*}
% for each fixed $t \in \C$ and each $R>1$.
\end{rem}

%\medskip

\begin{proof}[Proof of Corollary \ref{cor1}]
We use the terminology of Paulsen's book \cite{paul}. Let $\rho : R(X) \mapsto \mathcal{L}(H)$ be the homomorphism given by $\rho(f) = f(A)$. Theorem \ref{thm1} implies that the complete bounded norm $\|\rho\|_{cb}$ of $\rho$ is bounded by $2+2/\sqrt{3}$. Using a theorem of Paulsen \cite[Theorem 9.1]{paul}, there exists an invertible operator $L$ with $\|L\|\cdot\|L^{-1}\| = \|\rho\|_{cb} \le 2+2/\sqrt{3}$ such that $L^{-1}\rho(\cdot)L$ is a unital completely contractive homomorphism. Thus $X$ is a complete spectral set for $L^{-1}AL$. Therefore, as a consequence of Arveson's extension theorem (see \cite[Corollary 7.8]{paul}), $L^{-1}AL$ has a normal dilation with spectrum included in $\partial X$, as claimed in Corollary \ref{cor1}.
\end{proof}

%\medskip
\begin{rem}
According to a deep result due to Agler \cite{agler}, if $X$ is a spectral set for $A$, then $X$ is a complete spectral set for $A$, and thus $A$ has a normal dilation with spectrum included in $\partial X$.
%is in fact a complete spectral set for $L^{-1}AL$. 
The analogue of Agler's theorem is not true for triply connected domains (see \cite{DM}).
\end{rem}

\bibliographystyle{amsplain}

\bigskip

\noindent
Laboratoire Paul Painlev\' e,
UMR CNRS no. 8524,\\
Universit\'e de Lille 1,
59655 Villeneuve d'Ascq Cedex, France
\\
{\tt Catalin.Badea@math.univ-lille1.fr}\\

\noindent
Laboratoire Paul Painlev\' e,
UMR CNRS no. 8524,\\
Universit\'e de Lille 1,
59655 Villeneuve d'Ascq Cedex, France
\\
{\tt Bernhard.Beckermann@math.univ-lille1.fr}\\

\noindent Institut de Recherche Math\'ematique de Rennes, UMR 6625 au CNRS,\\
Universit\'e de Rennes 1,
Campus de Beaulieu, 35042 RENNES Cedex, France\\
{\tt michel.crouzeix@univ-rennes1.fr}
\end{document}